\documentclass[a4paper]{amsart}

\usepackage{amsmath,amssymb,amsthm}
\usepackage{indentfirst,color}
\usepackage[colorlinks,citecolor=red,linkcolor=blue,urlcolor=cyan]{hyperref}

\theoremstyle{plain}
\newtheorem{theorem}{Theorem}[section]
\newtheorem{proposition}[theorem]{Proposition}
\newtheorem{lemma}[theorem]{Lemma}
\newtheorem{corollary}[theorem]{Corollary}
\theoremstyle{definition}
\newtheorem{definition}[theorem]{Definition}
\theoremstyle{remark}
\newtheorem{remark}[theorem]{Remark}
\numberwithin{equation}{section}
\newtheorem{conj}[theorem]{Conjecture}
\newtheorem{example}[theorem]{Example}

\newcommand{\be}{\begin{equation}}
\newcommand{\ee}{\end{equation}}
\newcommand{\bea}{\begin{eqnarray}}
\newcommand{\eea}{\end{eqnarray}}
\newcommand{\ben}{\begin{eqnarray*}}
	\newcommand{\een}{\end{eqnarray*}}
\newcommand{\bt}{\begin{split}}
	\newcommand{\et}{\end{split}}
\newcommand{\bet}{\begin{equation}}
\newcommand{\CC}{\mathbb{C}}
\newcommand{\DD}{\mathbb{D}}
\newcommand{\RR}{\mathbb{R}}

\newcommand{\beq}{\begin{equation*}}
\newcommand{\eeq}{\end{equation*}}
\newcommand{\bal}{\begin{aligned}}
\newcommand{\eal}{\end{aligned}}

\newcommand{\ddbar}{\partial \bar{\partial}}
\newcommand{\dbar}{\bar{\partial}}

\newcommand{\Id}{{\textup{Id}}}%
\newcommand{\supp}{{\textup{Supp}\,}}%
\newcommand{\loc}{{\textup{loc}}}%
\newcommand{\Ker}{{\textup{Ker}}}

\newcommand{\eps}{\varepsilon}%
\renewcommand{\leq}{\leqslant}%
\newcommand{\inner}[1]{\langle#1\rangle}

\begin{document}

\title[Optimal $L^2$ Extension]{An optimal $L^2$ extension for continuous $L^2$-optimal Hermitian metrics}

\author{Zhuo Liu}
\address{Mathematical Science Research Center,  Chongqing University of Technology,  No. 69, Hongguang
Avenue, Banan District, Chongqing 400054, China.}
\email{liuzhuo@cqut.edu.cn; liuzhuo@amss.ac.cn}

\date{}

\begin{abstract}
In this paper, we obtain an optimal $L^2$ extension theorem for  continuous $L^2$-optimal Hermitian metric on bounded planer domains. As applications, we affirmatively answer a question of Deng-Ning-Wang and a question of Inayama.
\end{abstract}

\keywords{$L^2$-optimal metrics,   Optimal $L^2$ extension, Pluriharmonicity, Griffiths positivity}

\subjclass[2020]{32U05, 32D15, 32L15, 32U35}

\maketitle

\section{Introduction}

Plurisubharmonicity and Griffiths/Nakano positivity play fundamental roles in several complex variables and complex geometry, yielding numerous important results. Plurisubharmonic functions, which need not be smooth, offer significant advantages in many problems. Similarly, there is growing interest in singular metrics with positivity properties for vector bundles.

Let $E$ be a holomorphic vector bundle over a complex manifold $X$. A \textit{singular Hermitian metric} $h$ on $E$ is a measurable map from $X$ to the space of non-negative Hermitian forms on fibers, satisfying $0 < \det h < +\infty$ almost everywhere. When $h$ is smooth, it is Griffiths semi-positive if and only if $\log|u|_{h^*}$ is plurisubharmonic for any local holomorphic section $u$ of the dual bundle. This characterization naturally leads to a definition of Griffiths positivity for singular Hermitian metrics (\cite{BP08, PT18, Raufi}), which has proven to be very useful. For line bundles, $h$ is Griffiths semi-positive precisely when the local weight $\varphi := -\log h$ is plurisubharmonic.

Recall that, for plurisubharmonic functions and Nakano semi-positive Hermitian holomorphic vector bundles on pseudoconvex domains or Stein manifolds, there are H\"ormander's $L^2$ estimates for the $\overline{\partial}$-equations \cite{Dem-82, Hor65} and Ohsawa-Takegoshi's $L^2$ extension theorem \cite{OT87}.
Since  \cite{OT87}, there has been considerable interests in refining the estimate in Ohsawa-Takegoshi's $L^2$ extension theorem. After the breakthrough of Guan-Zhou-Zhu \cite{GZZ11}, B{\l}ocki \cite{Blocki13} and Guan-Zhou \cite{GuanZhou15} successfully obtained the optimal $L^2$ extension theorem.

Recently, Deng-Ning-Wang-Zhang-Zhou \cite{DNW21, DNWZ22, DWZZ18} developed a converse $L^2$ theory, characterizing plurisubharmonicity and Griffiths/Nakano positivity via $L^2$-conditions for $\overline{\partial}$.

\begin{definition}[\cite{DNWZ22, Ina-AG, LXYZ24}]
Let $D \subset \CC^n$ be a domain, $E$ a trivial holomorphic vector bundle, and $h$ a singular Hermitian metric on $E$. The triple $(D, E, h)$ is \emph{$L^2$-optimal} if for any $\dbar$-closed $E$-valued $(n,1)$-form $f \in L^2_{(n,1)}(D, E; \loc)$, for any smooth strictly plurisubharmonic function $\phi$ and K\"ahler metric $\omega$ on $D$, the equation $\dbar u = f$ admits a solution  satisfying:
   \begin{equation}\label{eq:a1}
   \int_D |u|^2_{\omega,h} e^{-\phi} dV_\omega \leq \int_D \inner{B_{\omega,\phi}^{-1}f,f}_{\omega,h} e^{-\phi} dV_{\omega},
  \end{equation}
provided the right-hand side is finite, where $B_{\omega,\phi} := [i\ddbar \phi \otimes \Id_E, \Lambda_{\omega}]$.

A singular Hermitian metric $h$ on $E$ over a complex manifold $X$ is \emph{$L^2$-optimal} if for every Stein coordinate neighborhood $U \subset X$ with $E|_U$ trivial, $(U,E|_U,h)$ is $L^2$-optimal. When $E$ is a line bundle with $h = e^{-\varphi}$, $\varphi$ is \emph{$L^2$-optimal} if $h$ is, and $(D,\varphi)$ is \emph{$L^2$-optimal} if $(D,E,h)$ is.
\end{definition}

\begin{remark}
The integrals in \eqref{eq:a1} are independent of $\omega$ since $u$ is an $(n,0)$-form and $f$ an $(n,1)$-form.
\end{remark}

Deng-Ning-Wang \cite[Theorem 1.2]{DNW21} established that $C^2$-smooth $L^2$-optimal functions are plurisubharmonic, and asked whether continuity is sufficient. Separately, Deng-Ning-Wang-Zhou \cite[Theorem 1.1]{DNWZ22} showed $C^2$-smooth $L^2$-optimal Hermitian metrics are Nakano semi-positive. Since Nakano positivity implies Griffiths positivity for smooth metrics, Inayama \cite{Ina-AG} questioned whether $L^2$-optimality implies Griffiths semi-positivity for singular metrics.
We thus reformulate these questions as:

\begin{conj}[{\cite[Remark 1.9]{DNW21},\cite[Question 7.3]{Ina-AG}}]\label{conj DNW-Ina}
   Every continuous $L^2$-optimal Hermitian metric is Griffiths semi-positive.
\end{conj}

In this paper, adapting B\l ocki's method \cite{Blocki13}, we establish an optimal $L^2$-extension theorem for continuous $L^2$-optimal metrics on bounded planar domains as follows:

\begin{theorem}\label{thm optimal l2 ext}
 Let $D\subset\CC$ be a bounded domain, $E$  a   holomorphic vector bundle and $h$ a continuous Hermitian metric on $E$.  If $(D,E,h)$ is $L^2$-optimal, then for any $w \in D$ and $s \in E_w$, there exists $f \in H^0(D,E)$ with $f(w) = s$ satisfying
$$\int_D |f|^2_h  d\lambda_1 \leq \frac{\pi |s|^2_{h(w)}}{c_D(w)^2},$$
where $d\lambda_1$ is the Lebesgue measure of $\mathbb{C}$ and $c_D$ is the logarithmic capacity of $D$.
\end{theorem}

Note that continuous $L^2$-optimal Hermitian metrics remain $L^2$-optimal when restricted to complex hyperplanes (Lemma \ref{pro nak res nak}). Since Griffiths semi-positivity is characterized by restrictions to complex lines (Lemma \ref{lem ch Grif line}) and the optimal $L^2$-extension property implies Griffiths semi-positivity (\cite[Theorem 1.3]{DNWZ22}),  Theorem \ref{thm optimal l2 ext} resolves Conjecture \ref{conj DNW-Ina}:

\begin{theorem}\label{thm DNW-Ina}
Every continuous $L^2$-optimal Hermitian metric is Griffiths semi-positive.
\end{theorem}

\begin{remark}
The continuity assumption can be relaxed to: (strong) upper semi-continuity on $D$ and continuity on $D \setminus P$, where $P \subset D$ is a closed pluripolar subset.
\end{remark}

Recall notations from \cite[\S 2.1]{Bern15} and \cite[\S 4.1]{DNWZ22}. Consider a projective fibration $\pi: Y^{n+m} \to X^n$ of complex manifolds and a holomorphic vector bundle $E \to Y$ with continuous Hermitian metric $h$. Denote by $F$ the \textit{ relative  direct image sheaf} $ \pi_*(K_{Y/X} \otimes E)$. If $\dim H^0(Y_x, K_{Y_x} \otimes E|_{Y_x})$ is constant for all $x \in X$, where $Y_x := \pi^{-1}(x)$, then:
\begin{enumerate}
    \item $F$ is a holomorphic vector bundle with fiber $F_x \cong H^0(Y_x, K_{Y_x} \otimes E|_{Y_x})$
    \item The $L^2$-metric $h_F$ on $F$ is defined for $u, v \in F_x$ by:
    \[
    \langle u, v\rangle_{h_F} := \int_{Y_x} i^{m^2} \sum_{\alpha,\beta} h_{\alpha\bar{\beta}} u_\alpha \wedge \overline{v_\beta}
    \]
\end{enumerate}

Following \cite[Theorem 4.3]{DNWZ22} and Theorem \ref{thm DNW-Ina}, we obtain that

\begin{corollary}
	In the above setting, assume $h$ is $L^2$-optimal. Then $h_F$ is continuous and locally $L^2$-optimal: for each $x \in X$, there exists a neighborhood $U_x \ni x$ such that $h_F|_{U_x}$ is $L^2$-optimal. Consequently, $h_F$ is Griffiths semi-positive. 
\end{corollary}

The remaining parts of this article are organized as follows. Section \ref{ss l2 minimal} establishes modified $L^2$ estimates for $L^2$-optimal triples; Section \ref{ss Green function} reviews basic properties of complex Green functions; Section \ref{ss opt l2 ext} proves Theorem \ref{thm optimal l2 ext}; and Section \ref{ss solution to conj DNW-Ina} resolves Theorem \ref{thm DNW-Ina}.

\textbf{Acknowledgement:}
The authors would like to thank Dr. Wang Xu for his helpful discussions.

\section{Preliminaries}  

\subsection{$L^2$ minimal solution}\label{ss l2 minimal}
In this section, we establish modified $L^2$ estimates for $L^2$-optimal triples twisted with continuous strictly plurisubharmonic functions using the $L^2$-minimal solution  and Richberg's global regularization theorem.

\begin{lemma}[Richberg's global regularization theorem, \cite{Richberg}]\label{Richberg}
  Let $\phi$ be a continuous plurisubharmonic function on $X$ which is
strictly plurisubharmonic on an open subset $D\subset X$, with $i\ddbar \phi \ge \lambda$ for some continuous
positive Hermitian $(1,1)$-form $\lambda$ on $D$. For any positive continuous function $\eps$ on $D$, there
exists a plurisubharmonic function $\phi_\eps$ in $C^0(X)\cap C^\infty(D)$  such that $\phi\le \phi_\eps\le \phi+\eps$ on $D$ and $\phi_\eps = \phi$ on $X\setminus D$, which is strictly plurisubharmonic on $D$ and satisfies $\ddbar \phi_\eps \ge (1-\eps)\lambda$.
\end{lemma}

Then by Lemma \ref{Richberg}, we can obtain that

\begin{proposition}\label{prop cont psh}
   Let $D$ be a bounded domain in $\CC$, $E$  a trivial holomorphic vector bundle and $h$ a continuous Hermitian metric on $E$.  Let $\phi$ be a continuous strictly plurisubharmonic function $\phi$ with $i\ddbar\phi\ge\lambda$  for some continuous
positive Hermitian $(1,1)$-form $\lambda$ on $D$.
  Assume that $(D,E,h)$ is  {$L^2$-optimal}, then for   any K\"{a}hler metric $\omega$ on $D$,
   the equation $\dbar u=f$ can be solved on $D$ for any $\dbar$-closed $E$-valued $(n,1)$-form
   $f \in L^2_{(n,1)}(D,E;\loc)$
   with the estimate:
   \begin{equation*}
   \int_D|u|^2_{\omega,h} e^{-\phi} dV_\omega
   \leq
  \int_D \inner{B_\lambda^{-1}f,f}_{\omega,h} e^{-\phi} dV_{\omega},
  \end{equation*}
   provided that the right-hand side is finite, where $B_\lambda:=[\lambda, \Lambda_{\omega}]$.
\end{proposition}
\begin{proof}

  By  Lemma \ref{Richberg}, for any $\eps\in(0,1)$, there is  a smooth strictly plurisubharmonic function $\phi_\eps$ on $D$ such that $\phi\le\phi_\eps\le\phi+\eps$ and
  $i\ddbar\phi_\eps\ge(1-\eps)\lambda$. Since  $(D,E,h)$ is  {$L^2$-optimal}, then for any $\dbar$-closed $E$-valued $(n,1)$-form   $f \in L^2_{(n,1)}(D,E;\loc)$, for any K\"{a}hler metric $\omega$ on $D$, there is  $u_{\eps}\in L^2_{(n,0)}(D,E;\loc)$ such that
     $\dbar u_{\eps}=f$ and
   \begin{align*}
   \int_D|u_{\eps}|^2_{\omega,h} e^{-\phi} dV_\omega&\leq
   e^{\eps}\int_D|u_{\eps}|^2_{\omega,h} e^{-\phi_\eps} dV_\omega\\
   &\leq
  e^{\eps}\int_D \inner{B_{\phi_\eps}^{-1}f,f}_{\omega,h} e^{-\phi_\eps} dV_{\omega}\\
  &\leq e^{\eps}\int_D \inner{B_{(1-\eps)\lambda}^{-1}f,f}_{\omega,h} e^{-\phi_\eps} dV_{\omega}\\
  &\leq \frac{e^{\eps}}{1-\eps}\int_D \inner{B_{\lambda}^{-1}f,f}_{\omega,h} e^{-\phi} dV_{\omega}.
  \end{align*}
   Since $h$ is continuous, then $\{u_{\eps}\}$ is bounded in $L^2_{(n,0)}(D,E;\loc)$ and especially $\{u_{\eps}-u_{\eps_0}\}$ is bounded in $L^2_\loc(U)$. Since $u_{\eps}-u_{\eps_0}$ are holomorphic $(n,0)$-forms, then we can take a sequence $\varepsilon_k\to0$ such that
     $u_{\eps_k}-u_{\eps_0}$ compactly converging to a holomorphic $(n,0)$-form $u-u_{\eps_0}$ on $D$.
    Therefore, $\dbar u=f$ and by Fatou's lemma,
   we have
  \begin{align*}
     \int_D|u|^2_{\omega,h} e^{-\phi} dV_\omega
   &\le \liminf_{\varepsilon_k\to 0}\int_D|u_{\eps_k}|^2_{\omega,h} e^{-\phi_\eps} dV_\omega\\
     & \le \liminf_{\varepsilon_k\to 0} \frac{e^{\eps_k}}{1-\eps_k}\int_D \inner{B_{\lambda}^{-1}f,f}_{\omega,h} e^{-\phi} dV_{\omega}\\
     & \le \int_D \inner{B_{\lambda}^{-1}f,f}_{\omega,h} e^{-\phi} dV_{\omega}.
  \end{align*}
\end{proof}

 \begin{theorem}[{\cite[Theorem 2]{Blocki13}}]\label{thm L2 estimate by minimal solution}
 Let $D$ be a bounded domain in $\CC$, $E$  a trivial holomorphic vector bundle and $h$ a continuous Hermitian metric on $E$. Assume that $(D,E,h)$ is $L^2$-optimal. Let $\phi$ be a bounded, continuous and strictly plurisubharmonic function with $i\ddbar\phi\ge\lambda$  for some continuous
positive Hermitian $(1,1)$-form $\lambda$ on $D$. Let $\psi\in C^{1,1}(D)$ is  bounded from above on $D$ and  satisfy $$ i\partial\psi\wedge\dbar\psi \le H\lambda$$ on $D$   for some $H\in L^\infty(D)$ with $0<H<1$.
Let $f\in L^2_{(n,1)}(D,E;\loc)$ be $\dbar$-closed  such that
$$\int_D  \inner{B_\lambda^{-1}f,f}_{\omega,h} dV_{\omega}<+\infty.$$ Then there exists $u\in L^2_{(n,0)}(D,E;\loc)$ solving $\dbar u=f$ such that for every $b>0$,
 $$\int_D|u|^2_{\omega,h}(1 - H)e^{2\psi-\phi}dV_\omega\le
 (1+ \frac{1}{b})\int_D  \inner{B_\lambda^{-1}f,f}_{\omega,h}\frac{1+bH}{1-H} e^{2\psi-\phi} dV_{\omega}.\,$$
  provided that the right-hand side is finite.
 If in addition $H \le c<1$ on $\supp f$, then
 $$\int_D|u|^2_{\omega,h}(1 - H)e^{2\psi-\phi}dV_\omega\le
   \frac{1+\sqrt{c}}{1-\sqrt{c}}\int_D  \inner{B_\lambda^{-1}f,f}_{\omega,h}  e^{2\psi-\phi} dV_{\omega}.$$
 \end{theorem}

  \begin{proof}

  It follows from   Proposition \ref{prop cont psh} that  there is  $u\in L^2_{(n,0)}(D,E;\loc)$ such that
     $\dbar u=f$ and
   \begin{align*}
   \int_D|u |^2_{\omega,h} e^{-\phi} dV_\omega
  &\leq \int_D \inner{B_{\lambda}^{-1}f,f}_{\omega,h} e^{-\phi} dV_{\omega}.
  \end{align*}
Since $\psi$ is bounded from above on $D$,  we can take $u$ to be the minimal solution to $\dbar u=f$ in $L^2_{(n,0)}(D,E;he^{\psi-\phi},\omega)$, which
 means that $u $ is perpendicular to $\Ker~ \dbar$ in $L^2_{(n,0)}(D,E;he^{\psi-\phi},\omega)$.
Then $ ue^\psi$ is perpendicular
 to $\Ker~ \dbar$ in $L^2_{(n,0)}(D,E;he^{-\phi},\omega)$. Denote  $\beta :=\dbar(ue^\psi)=e^\psi(f+u\dbar \psi)$. By the Cauchy-Schwartz inequality, we get for any $r>0$,
 \begin{align*}
 & \int_D\inner{B_{\lambda}^{-1}\beta,\beta}_{\omega,h}e^{-\phi}dV_\omega \\
  \le &\int_D\inner{B_{\lambda}^{-1} (f+u\dbar \psi),f+u\dbar \psi}_{\omega,h}e^{2\psi-\phi}dV_\omega\\
     \le&  \int_D((1+r)\inner{B_{\lambda}^{-1}f,f}_{\omega,h}
     +(1+\frac{1}{r})H|u|^2_{\omega,h}) e^{2\psi-\phi}dV_\omega<+\infty,
  \end{align*}
  where the second inequality is owing to the fact   $\lambda\ge i H^{-1}\partial\psi\wedge\dbar\psi$. Since $(D,E,h)$ is  {$L^2$-optimal},  by Proposition \ref{prop cont psh}, we can solve $\dbar v= \beta$ in $L^2_{(n,0)}(D,E;he^{-\phi},\omega)$ such that
   \begin{align*}
   \int_D|v |^2_{\omega,h} e^{-\phi} dV_\omega
  &\leq \int_D \inner{B_{\lambda}^{-1}\beta,\beta}_{\omega,h} e^{-\phi} dV_{\omega}.
  \end{align*} Noticing that  $ ue^\psi$ is perpendicular
 to $\Ker~ \dbar$ in $L^2_{(n,0)}(D,E;he^{-\phi},\omega)$, we know that
  it is  the minimal solution to $\dbar v= \beta$  in $L^2_{(n,0)}(D,E;he^{-\phi},\omega)$.
 Therefore,
 \begin{align*}
 \int_D| ue^\psi|^2_{\omega,h}e^{-\phi}dV_\omega     \le &  \int_D\inner{B_{\phi}^{-1}\beta,\beta}_{\omega,h}e^{-\phi}dV_\omega\\
   \le&  \int_D((1+r)\inner{B_{\lambda }^{-1}f,f}_{\omega,h}
     +(1+\frac{1}{r})H|u|^2_{\omega,h}) e^{2\psi-\phi}dV_\omega.
 \end{align*}
  We obtain the first estimate if we take  $r := \frac{(1+b)H}{1-H}$,
   whereas the second one
 follows if we let $b := c^{-1/2}$.
  \end{proof}
\subsection{Complex Green's function on bounded planar domains}\label{ss Green function}

Let $D \subset \mathbb{C}$ be a domain. The \emph{complex Green function} $g_D(z, w)$ for $D$ with pole at $w \in D$ is defined as the  function satisfying:

\begin{enumerate}
    \item $g_D(\cdot, w)$ is harmonic on $D \setminus \{w\}$
    \item $g_D(z, w) \sim   \log |z - w|$ as $z \to w$ (logarithmic singularity)
    \item $g_D(z, w) = 0$ for $z \in \partial D$ (Dirichlet boundary condition)
\end{enumerate}

 Assume $D$ is bounded, then the complex Green function $g_D(z, w)$ exists and is unique. The function
\[
v_w(z) := g_D(z, w) - \log|z - w|
\]
is harmonic on $D$ and bounded by the maximum principle. The \emph{logarithmic capacity} at $w$ is defined as
\[
c_D(w) := \exp\left( \lim_{z \to w} \big( g_D(z, w) - \log|z - w| \big) \right).
\]
 
\begin{example}\label{exa Dr}
  Let $\mathbb{D}_r = \{ z \in \mathbb{C} : |z| < r \}$,  then
\begin{align*}
  g_{\mathbb{D}_r}(z, w) =& \log \left| \frac{r(z - w)}{r^2 - \overline{w}z} \right|\\
  c_{\mathbb{D}_r}(w)= & \frac{r}{r^2-|z|^2}.
\end{align*}
\end{example}

\section{Optimal $L^2$ extension}\label{ss opt l2 ext}

In this section, we obtain an optimal $L^2$-extension theorem for continuous $L^2$-optimal metrics on planar domains, adapting B\l ocki's method \cite{Blocki13} with minor modifications.

 \begin{theorem}[=Theorem \ref{thm optimal l2 ext}]
 Let $D\subset\CC$ be a bounded domain, $E$  a trivial holomorphic vector bundle and $h$ a continuous Hermitian metric on $E$.  If $(D,E,h)$ is $L^2$-optimal, then for any $w \in D$ and $s \in E_w$, there exists $f \in H^0(D,E)$ with $f(w) = s$ satisfying
$$\int_D |f|^2_h  d\lambda_1 \leq \frac{\pi |s|^2_{h(w)}}{c_D(w)^2},$$
where $d\lambda_1$ is the Lebesgue measure of $\mathbb{C}$ and $c_D$ is the logarithmic capacity of $D$.
\end{theorem}

\begin{proof}
Without loss of generality, we may assume that $w=0\in D$ and $s$ is a global section of $E$ on $D$.
  Fix   $0<\eps,\sigma\ll 1$, and define
  \begin{equation}\label{define alpha}
    \alpha=\alpha_{\eps,\sigma}:=s\dbar (\chi_{\eps,\sigma}(-\log|z|^2))dz=\frac{s\chi_{\eps,\sigma}'(-\log|z|^2)}{\bar z}dz\wedge d\bar z
  \end{equation}
 where $\chi_{\eps,\sigma}\in C^{0,1}(\RR)$ is non-decreasing and such that
 $\chi_{\eps,\sigma} =0$ on $\{t \le -2\log \eps\}$, $\chi_{\eps,\sigma} =1$ on $\{t \ge -2\log (\sigma\eps)\}$ (it will be precisely determined later).
 We want to solve  $\dbar u= \alpha$ under  appropriate weights $\phi$, $\psi$.

 Take  the following two functions on $\RR^+$:
 \begin{align*}
    \tau(t) & := -\log (t +e^{-t} -1), \\
   \rho(t) & :=-\log (t +e^{-t} -1)+\log(1-e^{-t}).
 \end{align*}

One can check that they are decreasing, convex,
 \begin{equation}\label{define tau rho}
   \left(1-\frac{(\rho')^2}{\tau''}\right)e^{2\rho-\tau+t}=1.
 \end{equation}
 In addition, $\frac{(\rho')^2}{\tau''}$ is strictly increasing, $\frac{(\rho')^2}{\tau''}(0)=\frac{1}{2}$ and $\lim\limits_{t\to+\infty}\frac{(\rho')^2}{\tau''}(t)=1$,
and
 \begin{equation}\label{property tau rho}
   \lim\limits_{t\to+\infty}\tau(t)-2\rho(t)+\log(-\tau'(t))=0.
 \end{equation}

 Let $G(z):=g_D(z,0)$ be the complex Green function for $D$ with pole at $0$, then
 \begin{equation}\label{Green function}
   v(z):=G(z)-\log|z|
 \end{equation}
  is a bounded harmonic function and
$c_D(0)=e^{v(0)}$ is the logarithmic capacity at $0$. For $\delta\in(0,\sigma\eps)$,
  we define $$G_\delta(z):=v(z)+\frac{1}{2}\log(|z|^2+\delta^2)-\sup_{z\in \partial D}\log(1+\frac{\delta^2}{|z|^2}).$$
 Then $G_\delta$ is smooth and strictly plurisubharmonic. Since $G(z)\equiv0$ on $\partial D$, by the maximum principle, we have
 \begin{equation}\label{G_gamma}
   \sup_{z\in D} G_\delta(z)=\sup_{z\in \partial D}G_\delta(z)\le -\frac{1}{2}\sup_{z\in \partial D}\log(1+\frac{\delta^2}{|z|^2})<0.
 \end{equation}
 Moreover, there are positive constants $C_0, C_1$ independent of $\delta$ and $\eps$ such that
 $$\left|2G_\delta(z)-\log(|z|^2+\delta^2)\right| \le C_0 {\textup {~on~}} D$$
 and
 $$\left|2\frac{\partial G_\delta}{\partial z}(z)- \frac{\bar z}{|z|^2+\delta^2}\right|\le C_1{\textup{~near~the~origin~}}0.$$
 Let $M :=-\log(2\eps^2)-C_0$, then $\supp\alpha\subset\{-2G_\delta>M\}$. We define
 \begin{equation*}
   \eta(t):=\left\{
   \begin{aligned}
     &\tau(t), &~~ 0<t<M,\\
     &-c\log(t-M+a)+b, &~~ t\ge M,
   \end{aligned}\right.
 \end{equation*}
 and
 \begin{equation*}
   \gamma(t):=\left\{
   \begin{aligned}
     &\rho(t), &~~ 0<t<M,\\
     &-c\log(t-M+\tilde a)+\tilde b, &~~ t\ge M,
   \end{aligned}\right.
 \end{equation*}
 where $c=M^{-1/2}$ and $a, b, \tilde{b}$ are chosen in such a way that $\eta, \rho\in C^{1,1}(\RR^+)$.
 Then \begin{align*}
        a =&a(\eps)=-\frac{c}{\tau'(M)}, \\
        b = &b(\eps)=\tau(M)+c\log a, \\
        \tilde{a}=&\tilde a(\eps)=-\frac{c}{\rho'(M)},\\
        \tilde b = &\tilde b(\eps)=\rho(M)+c\log \tilde a.
      \end{align*}
    One can check that $\tau'\le\rho'\le\tfrac{1}{2}\tau'<0$, then $2a\ge\tilde a\ge a$. In addition, $c=c(\eps)=M^{-1/2}$ is selected so that
      $$\lim\limits_{\eps\to 0}c(\eps)=0, ~~~~\lim\limits_{\eps\to 0}a(\eps)=+\infty.$$

 Now we take
$$\phi:=2G_\delta+\eta(-2G_\delta),\quad\quad \psi:=\gamma(-2G_\delta),$$
then $\phi\in C^{1,1}(\overline D)$ is strictly plurisubharmonic on $D$ and $\psi\in C^{1,1}(\overline D)$.
We have
\begin{align*}
  i\partial\psi\wedge\dbar\psi&=i(\gamma'(-2G_\delta))^2 \partial(-2G_\delta)\wedge\dbar(-2G_\delta),\\
  i\ddbar\phi&\ge i\eta''(-2G_\delta)\partial(-2G_\delta)\wedge\dbar(-2G_\delta)=:\lambda.
\end{align*}
Consider
\begin{equation*}
  0<H=\frac{(\gamma'(-2G_\delta))^2 }{\eta''(-2G_\delta)}=\left\{\begin{aligned}
    &\frac{(\rho'(-2G_\delta))^2}{\tau''(-2G_\delta)}<1\quad &{\textup{ on~}} \{-2G_\delta\le M\} ,\\
    &c\left(1-\frac{\tilde a-a}{-2G_\delta-M+\tilde a}\right)^2<c \quad &{\textup{ on~}} \{-2G_\delta> M\}.
  \end{aligned}\right.
\end{equation*}
Note that $\frac{(\rho')^2}{\tau''}$ is strictly increasing and $\lim\limits_{t\to+\infty}\frac{(\rho')^2}{\tau''}(t)=1$,
and $c\left(1-\frac{\tilde a-a}{t-M+\tilde a}\right)^2$ is strictly increasing.
Hence we can find a  sequence $\{\theta_j\}$ of continuous functions on $\RR^+$ decreasing converging to $\frac{(\gamma')^2}{\eta''}$ such that  $\theta_j=\frac{(\gamma')^2}{\eta''}$ on  $\{t\le M\}\cup\{t\ge M+\frac{1}{j}\}$ and $\theta_j$ is decreasing on $\{M<t<M+\frac{1}{j}\}$.
Then $\frac{(\gamma')^2}{\theta_j}$ is continuous and  increasing converging to $\eta''$ with $\frac{(\gamma')^2}{\theta_j}=\eta''$ on  $\{t\le M\}\cup\{t\ge M+\frac{1}{j}\}$.
 Take $H_j:=\theta_j(-2G_\delta)$ and $$\lambda_j:=i\frac{(\gamma'(-2G_\delta))^2}{\theta_j(-2G_\delta)}\partial(-2G_\delta)\wedge\dbar(-2G_\delta),$$ then  $i\partial\psi\wedge\dbar\psi= iH_j\lambda_j$, $\lambda_j\le\lambda\le i\ddbar\phi$.

Since  the function $-2c\log(t-M+\tilde a)+c\log(t-M+a)+t$ is increasing in $t$  for sufficiently small $\eps$, together with \eqref{define tau rho}, we have
\begin{equation*}
  (1- H_j)e^{2\psi-\phi}\left\{
  \begin{aligned}
    &=(1- H)e^{2\psi-\phi}=1\quad&{\textup{on~}}\{-2G_\delta< M\},\\
    &\ge (1-H_j(M))e^{2\rho(M)-\tau(M)+M}=1\quad&{\textup{on~}}
    \{-2G_\delta\ge M\}.
  \end{aligned}\right.
\end{equation*}

Assume that \begin{align}\label{for claim}
               \int_D\inner{B^{-1}_{\lambda} \alpha,\alpha}_{\omega,h}\frac{1+bH}{1-H}e^{2\psi-\phi}dV_{\omega} <+\infty.
            \end{align}
 Notice that $c\lambda\le c\le \lambda_j$ and $\sup_{z\in D}H_j=\theta_j(M)=\frac{(\rho')^2}{\tau''}(M)<1$, then
Theorem \ref{thm L2 estimate by minimal solution} now gives a solution $u_{\eps,\sigma,\delta,j}$ of $\dbar u_{\eps,\sigma,\delta,j}=\alpha$  such that for any $b>0$
\begin{align*}
 \int_D|u_{\eps,\sigma,\delta,j}|^2_{\omega,h} dV_\omega &\le  \int_D|u_{\eps,\sigma,\delta,j}|^2_{\omega,h} (1- H_j)e^{2\psi-\phi}dV_\omega \\
   & \le(1+\frac{1}{b})\int_D\inner{B^{-1}_{\lambda_j} \alpha,\alpha}_{\omega,h}\frac{1+bH_j}{1-H_j}e^{2\psi-\phi}dV_{\omega}
\end{align*}
  Moreover,  we can take a sequence $j_k\to0$ such that $u_{\eps,\sigma,\delta,j_k}$ is compactly convergent to a limit $u_{\eps,\sigma,\delta}$ on $D$. Then  $\dbar u_{\eps,\sigma,\delta}=f$ on $D$ and by the dominated convergence theorem, we have
\begin{align}\label{integral for u}
 \int_D|u_{\eps,\sigma,\delta}|^2_{\omega,h} dV_\omega &\le \int_D|u_{\eps,\sigma,\delta}|^2_{\omega,h}  (1- H)e^{2\psi-\phi} dV_\omega\nonumber\\
   & \le(1+\frac{1}{b})\int_D\inner{B^{-1}_{\lambda} \alpha,\alpha}_{\omega,h}\frac{1+bH}{1-H}e^{2\psi-\phi}dV_{\omega}.\nonumber\\
   &\le\frac{1+\sqrt{c}}{1-\sqrt{c}}\int_{\supp\alpha}\inner{B^{-1}_{\lambda} \alpha,\alpha}_{\omega,h}e^{2\psi-\phi}dV_{\omega}\nonumber\\
   &\le\frac{1+\sqrt{c}}{1-\sqrt{c}}\int_{\{\sigma\eps\le|z|\le\eps\}}\frac{(\chi_{\eps,\sigma}'(-2\log|z|))^2|s|^2_he^{2\psi-\phi}}{4\left|\tfrac{\partial G_\delta}{\partial z}(z)\right|^2|z|^2\eta''(-2G_\delta)}idz\wedge d\bar z,
\end{align}
where the second inequality is owing to the fact that $H\le c$ on $\supp \alpha$ and by taking $b=c^{-1/2}$, and the last inequality is due to the fact that on $\supp \alpha\subset\{-2G_\delta> M\}$,
 \begin{align*}
     \inner{B^{-1}_{\lambda} \alpha,\alpha}_{\omega,h}dV_{\omega}\le \frac{(\chi_{\eps,\sigma}'(-2\log|z|))^2|s|^2_{h}idz\wedge d\bar z}{4\left|\tfrac{\partial G_\delta}{\partial z}(z)\right|^2|z|^2\eta''(-2G_\delta)}.
    \end{align*}

Therefore, it suffices to estimate an upper bound for the integral on the right-hand side of \eqref{integral for u}.

Note that $$e^{-2G_\delta(z)}=\frac{\sup_{z\in \partial D}(1+\tfrac{\delta^2}{|z|^2})^2}{(|z|^2+\delta^2)e^{2v(z)}},$$
and on $\{\sigma\eps\le|z|\le\eps\}$, we have
$$16\left|\tfrac{\partial G_\delta}{\partial z}(z)\right|^2|z|^4\ge4\left|\tfrac{\partial G_\delta}{\partial z}(z)\right|^2(|z|^2+\delta^2)^2\ge\left||z|-C_1(|z|^2+\delta^2)
\right|^2\ge|z|^2|1-2C_1\eps|$$
(Notice that $\delta\in(0,\sigma\eps)$).
Moreover, since $2a>\tilde a >a$, one can check that $2\gamma-\eta$ is decreasing on $\{t>M\}$, and thus on $\{\sigma\eps\le|z|\le\eps\}$, we have
\begin{align*}
  2\psi-\eta(-2G_\delta)&\le2\gamma(-\log(|z|^2+\delta^2)-C_0)
-\eta(-\log(|z|^2+\delta^2)-C_0) \\
&\le 2\gamma(-\log(|z|^2+\delta^2))
-\eta(-\log(|z|^2+\delta^2))+c\log(2+\tfrac{2C_0}{\tilde a}).
\end{align*}
In addition, since $\eta''$ is decreasing on $\{t>M\}$, we have  $$\eta''(-2G_\delta)\ge\eta''(-\log(|z|^2+\delta^2)+C_0)\ge
(\tfrac{a+C_0}{a+2C_0})^2\eta''(-\log(|z|^2+\delta^2)).$$

Note that $v,h$ is continuous at $0$, then there is a positive constant $C$ independent of $\eps,\sigma,\delta$ such that
\begin{align*}
  &\frac{1+\sqrt{c}}{1-\sqrt{c}}\int_{\{\sigma\eps\le|z|\le\eps\}}\frac{(\chi_{\eps,\sigma}'(-2\log|z|))^2|s|^2_{h}e^{2\psi-\phi}}{4\left|\tfrac{\partial G_\delta}{\partial z}(z)\right|^2|z|^2\eta''(-2G_\delta)}idz\wedge d\bar z  \\
  \le & C|s|^2_{h(0)}e^{-2v(0)}\int_{\{\sigma\eps\le|z|\le\eps\}}\frac{(\chi_{\eps,\sigma}'(-2\log|z|))^2
   e^{(2\gamma
-\eta)\circ(-\log(|z|^2+\delta^2))}}{ |z|^2\eta''(-\log(|z|^2+\delta^2))}idz\wedge d\bar z
\end{align*}

Note that on $\{t> M\}$,  \begin{align}\label{eta''exp(eta-2gamma)}
                            \eta''e^{\eta-2\gamma}=& ce^{b-2\tilde b}(t-M+\tilde a)^{2c}(t-M+a)^{-2-c}
                          \end{align} is decreasing in $t$ for $c$ small enough  and $-\log|z|^2>-\log(|z|^2+\delta^2)$, then we have

\begin{align}\label{integral for u A}
  &\frac{1+\sqrt{c}}{1-\sqrt{c}}\int_{\{\sigma\eps\le|z|\le\eps\}}\frac{(\chi_{\eps,\sigma}'(-\log|z|^2))^2|s|^2_{h}e^{2\psi-\phi}}{4\left|\tfrac{\partial G_\delta}{\partial z}(z)\right|^2|z|^2\eta''(-2G_\delta)}idz\wedge d\bar z  \nonumber\\
  \le & C|s|^2_{h(0)}e^{-2v(0)}\int_{\{\sigma\eps\le|z|\le\eps\}}\frac{(\chi_{\eps,\sigma}'(-\log|z|^2))^2
   e^{(2\gamma
-\eta)\circ(-\log|z|^2)}}{ |z|^2\eta''(-\log|z|^2) } idz\wedge d\bar z\nonumber\\
\le & 2C\pi |s|^2_{h(0)}e^{-2v(0)}\int_{-2\log\eps}^{{-2\log(\sigma\eps)}}\frac{(\chi_{\eps,\sigma}')^2}{
\eta''e^{\eta-2\gamma}}dt.
\end{align}

Now we take
\begin{equation*}
  \chi_{\eps,\sigma}(t):=\left\{
  \begin{aligned}
    &0,\quad& t<-2\log\eps,\\
    &\frac{1}{V_{\eps,\sigma}}\int_{-2\log\eps}^{t}\eta''e^{\eta-2\gamma}ds,\quad& -2\log\eps\le t\le-2\log(\sigma\eps),\\
    &1,\quad& t\ge-2\log(\sigma\eps),\\
  \end{aligned}\right.
\end{equation*}
where $V_{\eps,\sigma}=\int_{-2\log\eps}^{-2\log(\sigma\eps)}\eta''e^{\eta-2\gamma}dt$. Then $V_{\eps,\sigma}$ is decreasing with respect to $\sigma$.

Recall that $2a\ge\tilde{a}\ge a$ and $a=-\frac{c}{\tau'(M)}$, then by \eqref{eta''exp(eta-2gamma)},
  on $-2\log\eps\le t\le-2\log(\sigma\eps)$, we have
 
   $$\eta''e^{\eta-2\gamma}\le {2^{2c}a^{1-c}}e^{\tau(M)-2\rho(M)+\log(-\tau'(M))}(t-M+a)^{-2+c},$$
   
  and  
  
  $$ \eta''e^{\eta-2\gamma} \ge {2^{-2c}a^{1-c}}e^{\tau(M)-2\rho(M)+\log(-\tau'(M))}(t-M+ a)^{-2+c}.$$

  Since  \begin{align*}& \int_{-2\log\eps}^{-2\log(\sigma\eps)}a^{1-c}(t-M+a)^{-2+c}dt\\=&
  \frac{a^{1-c}}{1-c}(t-M+a)^{-1+c}\Big|_{-2\log(\sigma\eps)}^{-2\log\eps}\\
  =&\frac{a^{1-c}}{1-c}\left((\log2+C_0+a)^{-1+c}-(-2\log\sigma+\log2+C_0+a)^{-1+c}\right)
  \end{align*}

  Recall that $\lim_{\eps\to 0}c=0$, $\lim_{\eps\to 0}a=+\infty$ and $\lim_{\eps\to 0}M=+\infty$, then by \eqref{property tau rho}, we have
\begin{align}\label{integral for u B}
  \lim_{\eps\to 0} V_{\eps}:=\lim_{\eps\to 0}\left(\lim_{\sigma \to 0} V_{\eps,\sigma}\right)=1.
 \end{align}

In summary, by \eqref{integral for u}, \eqref{integral for u A} and \eqref{integral for u B}, for any $\eps,\sigma$ small enough, we obtain that
\begin{align*}
  \int_D|u_{\eps,\sigma,\delta}|^2_{\omega,h} dV_\omega &\le \int_D|u_{\eps,\sigma,\delta}|^2_{\omega,h}  (1- H)e^{2\psi-\phi} dV_\omega\\&\le  \frac{1+\sqrt{c}}{1-\sqrt{c}} \int_{\{\sigma\eps\le|z|\le\eps\}}\frac{(\chi_{\eps,\sigma}'(-\log|z|^2))^2|s|^2_he^{2\psi-\phi}}{4\left|\tfrac{\partial G_\delta}{\partial z}(z)\right|^2|z|^2\eta''(-2G_\delta)}idz\wedge d\bar z  \\
   & \le 2C\pi |s|^2_{h(0)}e^{-2v(0)}\int_{-2\log\eps}^{{-2\log(\sigma\eps)}}\frac{(\chi_{\eps,\sigma}')^2}{
\eta''e^{\eta-2\gamma}}dt\\
&\le \frac{2C\pi |s|^2_{h(0)}e^{-2v(0)}}{V_{\eps,\sigma}^2}\int_{-2\log\eps}^{{-2\log(\sigma\eps)}}\eta''e^{\eta-2\gamma}dt \\
&\le  4C\pi|s|^2_{h(0)}e^{-2v(0)}.
\end{align*}

Then  we can take a sequence $\delta_j\to0$ such that $u_{\varepsilon,\sigma,\delta_j}$ is compactly convergent to a limit $u_{\eps,\sigma}$ on $D$. Then  $\dbar u_{\eps,\sigma}=\alpha_{\eps,\sigma}$ on $D$ and by Fatou's lemma,  we have

\begin{align*}
 & \int_D|u_{\eps,\sigma}|^2_{\omega,h} dV_\omega \\\le& \int_D|u_{\eps,\sigma}|^2_{\omega,h}  \left(1- \frac{(\gamma'(-2G))^2}{\eta''(-2G)}\right)e^{2\gamma(-2G)-\eta(-2G)-2G} dV_\omega\\ \le& \liminf_{\delta_j\to0}\frac{1+\sqrt{c}}{1-\sqrt{c}}\int_{\{\sigma\eps\le|z|\le\eps\}}
 \frac{(\chi_{\eps,\sigma}'(-\log|z|^2))^2|s|^2_he^{2\psi-\phi}}{4\left|\tfrac{\partial G_{\delta_j}}{\partial z}(z)\right|^2|z|^2\eta''(-2G_{\delta_j})}idz\wedge d\bar z  \\
  \le & \frac{1+\tfrac{2\sqrt{c}}{1-\sqrt{c}}}{|1-2C_1\eps|}\cdot\left(2+\frac{2C_0}{\tilde a}\right)^c\cdot\left(\frac{a+2C_0}{a+C_0}\right)^2\int_{\{\sigma\eps\le|z|\le\eps\}}
   \frac{\eta''e^{\eta-2\gamma}|s|^2_h}{{V_{\eps,\sigma}^2}|z|^2e^{2v }}idz\wedge d\bar z
    \\ :=&C_\eps \int_{\{\sigma\eps\le|z|\le\eps\}}
   \frac{ \eta''e^{\eta-2\gamma}|s|^2_h}{{V_{\eps,\sigma}^2}|z|^2e^{2v }}idz\wedge d\bar z.
\end{align*}
 Note that $\lim_{\eps\to0}C_\eps=1$,
then $\int_D|u_{\eps,\sigma}|^2_{\omega,h} dV_\omega$ is uniformly bounded with respect to $\eps,\sigma$. Hence  we can take a sequence $\sigma_j\to0$ such that $u_{\varepsilon,\sigma_j}$   converge weakly to limits $u_{\varepsilon}$  on $D$. Moreover, by Mazur's theorem, there is a sequence $\{\widetilde{ u_{\eps,\sigma_j}}\}$ such that $\widetilde{ u_{\eps,\sigma_j}}$ strongly converges to $u_\eps$ where each $\widetilde{ u_{\eps,\sigma_j}}$ is a convex combination of
	$\{u_{\eps,\sigma_k} \}_{k\ge j}$. Write $\widetilde{ u_{\eps,\sigma_j}}=\sum_{m=1}^{N}a_mu_{\eps,\sigma_{k_m}}$ for $a_m>0$ with $\sum_{m=1}^{N}a_m=1$, then by the Cauchy-Schwarz inequality, we have
\begin{align*}
  &\int_{D }|\widetilde{ u_{\eps,\sigma_j}}|^2_{\omega,h} dV_\omega \\
  \le& \int_D|\widetilde{ u_{\eps,\sigma_j}}|^2_{\omega,h}  \left(1- \frac{(\gamma'(-2G))^2}{\eta''(-2G)}\right)e^{2\gamma(-2G)-\eta(-2G)-2G} dV_\omega\\
 \le &\int_{D }|\sum_{m=1}^{N}a_mu_{\eps,\sigma_{k_m}}|^2_{\omega,h} \left(1- \frac{(\gamma'(-2G))^2}{\eta''(-2G)}\right)e^{2\gamma(-2G)-\eta(-2G)-2G} dV_\omega\\
 \le &\int_{D }\left(\sum_{m=1}^{N}a_m\right)
 \left(\sum_{m=1}^{N}a_m|u_{\eps,\sigma_{k_m}}|^2_{\omega}\right)  \left(1- \frac{(\gamma'(-2G))^2}{\eta''(-2G)}\right)e^{2\gamma(-2G)-\eta(-2G)-2G}dV_\omega\\
 \le&  C_\eps\int_{\{\sigma_j\eps\le|z|\le\eps\}}
   \frac{\eta''e^{\eta-2\gamma}|s|^2_h}{{V_{\eps,\sigma_j}^2}|z|^2e^{2v }}idz\wedge d\bar z.
\end{align*}Therefore,  we have
\begin{align*}
 & \int_D|u_{\eps}|^2_{\omega,h} dV_\omega \\\le& \int_D|u_{\eps}|^2_{\omega,h}  \left(1- \frac{(\gamma'(-2G))^2}{\eta''(-2G)}\right)e^{2\gamma(-2G)-\eta(-2G)-2G} dV_\omega\\ \le& \liminf_{\sigma_j\to0}C_\eps\int_{\{\sigma_j\eps\le|z|\le\eps\}}
   \frac{\eta''e^{\eta-2\gamma}|s|^2_h}{{V_{\eps,\sigma_j}^2}|z|^2e^{2v }}idz\wedge d\bar z  \\
   \le&  C_\eps\int_{\{|z|\le\eps\}}
   \frac{\eta''e^{\eta-2\gamma}|s|^2_h}{{V_{\eps}^2}|z|^2e^{2v }}idz\wedge d\bar z
\end{align*}
Since for any $0<\tilde\eps\le\eps$ \begin{align*}
         &  \int_{|z|\le\tilde\eps}\left(1- \frac{(\gamma'(-2G))^2}{\eta''(-2G)}\right)e^{2\gamma(-2G)-\eta(-2G)-2G}idz\wedge d\bar z \\
         \ge &  (1-c)e^{-C_0}\int_{|z|\le\tilde\eps}\frac{e^{2\gamma(-2G)-\eta(-2G)}}{|z|^2}idz\wedge d\bar z\\
         \ge &2\pi(1-c)e^{-C_0}e^{2\tilde b-b}\int_{-2\log\tilde\eps}^{+\infty}(t-M+\tilde a)^{-2c}(t-M+a)^{c}dt\\
         \ge &2\pi(1-c)e^{-C_0}e^{2\tilde b-b}\int_{-2\log\tilde\eps}^{+\infty}(t-M+\tilde a)^{-c}dt\\
         =&+\infty,
      \end{align*}
we get that $u_{\eps}(0)=0$. Notice that $\dbar u_{\eps}=s\chi_{\eps}(-\log|z|^2)$ on $D$, then $$f_\eps:=s\chi_{\eps}(-\log|z|^2)-\frac{u_{\eps}}{dz}$$    is a holomorphic  function  on $D$ with $f_\eps(0)=s$, where \begin{equation*}
  \chi_{\eps }(t):=\left\{
  \begin{aligned}
    &0,\quad& t<-2\log\eps,\\
    &\frac{1}{V_{\eps }}\int_{-2\log\eps}^{t}\eta''e^{\eta-2\gamma}ds,\quad& -2\log\eps\le t,\\
    &1,\quad& t=+\infty.\\
  \end{aligned}\right.
\end{equation*} Moreover, for any $r>0$,

\begin{align*}
   & \int_D|f_\eps|^2_{h} idz\wedge d\bar z   \\
  \le & (1+\frac{1}{r})\int_D|\chi_{\eps}(-\log|z|^2)|^2|s|^2_{h} idz\wedge d\bar z +(1+r)\int_D|u_{\eps}|^2_{\omega,h}dV_\omega\\
  \le &  (1+\frac{1}{r})\int_{\{|z|\le\eps\}} |s|^2_{h} idz\wedge d\bar z +4(1+ {r})C\pi |s|^2_{h(0)}e^{-2v(0)}.
\end{align*}

Then we can take a sequence $\eps_j\to0$ such that $f_{\varepsilon_j}$ is compactly convergent to a  holomorphic  function  $f$ on $D$ with $f(0)=s$. In addition, by Fatou's lemma, we have

\begin{align*}
  &\int_D|f|^2_{h} idz\wedge d\bar z \\\le & \liminf_{\eps_j\to0} \int_D|f_{\eps_j}|^2_{h} idz\wedge d\bar z  \\\le& \liminf_{\eps_j\to0} (1+\frac{1}{r})\int_{\{|z|\le\eps_j\}} |s|^2_{h} idz\wedge d\bar z+(1+ {r})\int_D|u_{\eps_j}|^2_{\omega,h} dV_\omega\\ \le& \liminf_{\eps_j\to0}  (1+r)C_{\eps_j}\int_{\{ |z|\le\eps_j\}}
   \frac{\eta''e^{\eta-2\gamma}|s|^2_{h}}{V_{\eps_j}^2|z|^2e^{2v}}idz\wedge d\bar z  \\
   \le & (1+r) |s|^2_{h(0)}e^{-2v(0) } \liminf_{\eps_j\to0}\int_{\{ |z|\le\eps_j\}}
   \frac{\eta''e^{\eta-2\gamma}}{|z|^2} idz\wedge d\bar z\\
   \le & 2\pi(1+r) |s|^2_{h(0)}e^{-2v(0)}\liminf_{\eps_j\to0}\int_{-2\log\eps_j}^{+\infty}\eta''e^{\eta-2\gamma} dt\\
    \le & 2\pi(1+r) |s|^2_{h(0)}e^{-2v(0)}.
\end{align*}
  Since $r>0 $ is arbitrary, $v(0)=\log c_D(0)$ and $idz\wedge d\bar z=2d\lambda_1(z)$ we obtain that

  \begin{align*}
  \int_D|f|^2_{h}d\lambda_1(z)
    \le & \frac{\pi |s|^2_{h(0)}}{c_D(0)^2}.
\end{align*}

\end{proof}

\begin{remark}
\begin{enumerate}
  \item   The proof shows that for fixed $w \in D$, continuity of $h$ at $w$ is sufficient.
  \item  For any holomorphic vector bundle $E$ over a Stein manifold $X$ and $x \in X$, there exists an analytic subset $Z \subset X$ with $x \notin Z$ such that $E|_{X \setminus Z}$ is trivial. Since analytic subsets are $L^2$-negligible in the sense of \cite[Chapter VIII-(7.3)]{D12a}, the triviality assumption on $E$ over $D$ is unnecessary.
\end{enumerate}
\end{remark}

\section{Solution to Conjecture \ref{conj DNW-Ina}}\label{ss solution to conj DNW-Ina}

Recall that a singular Hermitian metric $h$ on a holomorphic vector bundle $E$ is said to be Griffiths semi-positive if   $\log|u|^2_{h^*}$ is plurisubharmonic for every local holomorphic section $u$ of $E^*$. Since upper semi-continuous functions are plurisubharmonic if and only if their restrictions to complex lines are subharmonic, we obtain:

\begin{lemma}\label{lem ch Grif line}
  Let $h$ be a singular Hermitian metric on a trivial holomorphic vector bundle $E$ over a domain $D\subset \CC^n$ such that $\log|u|^2_{h^*}$ is upper semi-continuous for any local holomorphic section $u$ of $E^*$. Then $h$ is Griffiths semi-positive  if and only if $h|_{L\cap D}$ is Griffiths semi-positive for any complex line $L\subset \CC^n$.
\end{lemma}

The following lemma reformulates \cite[Theorem 1.3]{DNWZ22} for singular Hermitian metrics on trivial bundles over planar domains:

\begin{lemma}[{\cite[Theorem 1.3]{DNWZ22}}]\label{lem DNWZ Grif}
   Let $h$ be a singular Hermitian metric on a trivial holomorphic vector bundle $E$ over a domain $D\subset \CC$ such that $\log|u|^2_{h^*}$ is upper semi-continuous for any local holomorphic section $u$ of $E^*$. Then $h$ is Griffiths semi-positive  if and only if $h$ satisfies the optimal $L^2$-extension property:  for any $w\in D$, any $\DD_r(w)\Subset D$ and $s\in E_w$, there is a holomorphic section $f\in H^0(\DD_r(w),E)$ with $f(w)=s$ and
 $$\frac{1}{\pi r^2}\int_{\DD_r(w)}|f|^2_hd\lambda_1(z)\le  |s|^2_{h(w)}.$$
\end{lemma}

Recall that \cite[Proposition 3.3]{LXYZ24}  established the restriction property: $L^2$-optimal Hermitian metrics remain $L^2$-optimal under restriction to complex hyperplanes. Due to minor differences in the setting, we provide a proof for completeness.

\begin{lemma}[Restriction property, {\cite[Proposition 3.3]{LXYZ24}}] \label{pro nak res nak}
Let $h$ be a continuous Hermitian metric on a trivial holomorphic vector bundle $E$ over a domain $D\subset \CC^n$.
	 Assume that $h$ is $L^2$ optimal, then for any complex hyperplane $H\subset \CC^n$,  $h|_{H\cap D}$ is also $L^2$-optimal.
\end{lemma}
\begin{proof}Without loss of generality, we may assume $H=\{(z',z_n)\in \CC^n;z_n=0\}$.
 Let $U\subset H\cap D$ be a Stein open subset,  $\omega$  a K\"{a}hler metric, $\phi$ a smooth strictly plurisubharmonic function on $U$ and $$f(z')=\sum_{j=1}^{n-1}f_j(z')d\bar z_j\in L^2_{(n-1,1)}(U, E; \loc)$$ a $\dbar$-closed $E$-valued $(n-1,1)$-form   such that
   \begin{equation*}
   \int_U \inner{B_{\omega,\phi}^{-1}f,f}_{\omega,h} e^{-\phi} dV_{\omega}<+\infty.
  \end{equation*}
Since we can take a Stein exhaustion of $U$ and regularize $f$ by convolutions, we may assume that $U\Subset H\cap D$ and $f$ is smooth.    Then for $\eps$ small enough, $$U_\eps:=U\times \DD_\eps=\{(z',z_n)\in \CC^n; z'\in U,|z_n|<\eps\}$$ is a Stein subset of $D$. In addition, $\widetilde\omega=\omega+\frac{i}{2}dz_n\wedge d\bar z_n$ is a K\"{a}hler metric on $U_\eps$, $\phi_\eps:=\phi+\eps|z_n|^2$ is a smooth strictly plurisubharmonic function on $U_\eps$ and $$\widetilde f(z',z_n)=\sum_{j=1}^{n-1}f_j(z',0)\wedge dz_n\wedge d\bar z_j$$ is a $\dbar$-closed  $E$-valued $(n,1)$-form on $U_\eps$.

Notice that $h$ is continuous and $$\langle B^{-1}_{\widetilde\omega,\phi_\eps} \widetilde f ,\widetilde f  \rangle_{\widetilde\omega} =\langle B^{-1}_{\omega,\phi}   f ,  f  \rangle_{ \omega},  $$

then  we obtain \begin{equation} \label{eq aaa}
  \lim_{\eps \to 0}\frac{1}{\pi \eps^2}\int_{U_{\eps}} \langle B^{-1}_{\widetilde\omega,\phi_\eps} \widetilde f ,\widetilde f  \rangle_{\widetilde\omega,h}e^{-\phi_\eps}dV_{\widetilde\omega}=\int_{U }\langle B^{-1}_{\omega,\phi}   f ,  f  \rangle_{ \omega,h}e^{-\phi}dV_{\omega}.
 \end{equation}

Since $h$ is $L^2$-optimal,
 there exists  $v_{\eps}\in L^2_{n,0}(U_{\eps},E;\loc)$ such that $\dbar v_{\eps}=\widetilde f$ and
     \begin{align*}
    \int_{U_{\eps}} |v_{\eps}|^2_{\omega,h} e^{-\phi_\eps} dV_{\widetilde\omega}
    \le \int_{U_{\eps}} \langle B^{-1}_{\omega,\phi_\eps}\widetilde f,\widetilde f \rangle_{\widetilde\omega,h} e^{-\phi_\eps} dV_{\widetilde\omega}.
 \end{align*}
 Moreover, by the weak regularity of $\dbar$ on $(n,0)$-forms, we can take $v_\eps$ to be smooth.
Using the Fubini-Tonelli theorem, we know that for any $\eps>0$, there exists  $\xi_\eps\in \DD_\eps$ such that \begin{equation*}
            \int_{U_\eps\cap\{z_n=\xi\} }|v_{\eps}(z',\xi_\eps)|^2_{\widetilde\omega,h}e^{-\phi_\eps}dV_{\widetilde\omega}\le
\frac{1}{\pi\eps^2}\int_{U_{\eps}} |v_{\eps}|^2_{\widetilde\omega,h} e^{-\phi_\eps} dV_{\widetilde\omega}.
          \end{equation*}
Let $u_{\eps}(z')=v_{\eps}(z',\xi_\eps)/dz_n$, then we have $\dbar u_{\eps}=f$ on $U$ and 
\begin{align}\label{eq bbb}
\int_{U  }|u_\eps|^2_{\omega,h(z',\xi_\eps)}e^{-\phi-\eps|\xi_\eps|^2}dV_{\omega}=&\int_{U  }|u_\eps|^2_{\widetilde\omega,h(z',\xi_\eps)}e^{-\phi-\eps|\xi_\eps|^2}dV_{\widetilde\omega}\nonumber\\ \le &
\frac{1}{\pi\eps^2}\int_{U_{\eps}} \langle B^{-1}_{\omega,\phi_\eps}\widetilde f,\widetilde f \rangle_{\widetilde\omega,h} e^{-\phi_\eps} dV_{\widetilde\omega}. 
\end{align}
Then we  can choose a sequence $\eps_k\to 0$ such that $u_{\eps_k}$  compactly converges to a limit $u$. Then   $\dbar u=f$ and by Fatou's lemma together with \eqref{eq aaa} and \eqref{eq bbb},  we have
\begin{align*}
  \int_U |u|^2_{\omega,h}e^{-\phi}dV_\omega \le &\liminf_{\eps_k\to0} \int_U |u_{\eps_k}|^2_{\omega,h}e^{-\phi-\eps_k|\xi_{\eps_k}|^2}dV_\omega\\
  \le & \liminf_{\eps_k\to0}\frac{1}{\pi\eps_k^2}\int_{U_{\eps_k}} \langle B^{-1}_{\omega,\phi_{\eps_k}}\widetilde f,\widetilde f \rangle_{\widetilde\omega,h} e^{-\phi_{\eps_k}} dV_{\widetilde\omega}\\
  \le & \int_{U }\langle B^{-1}_{\omega,\phi}   f,  f \rangle_{ \omega,h}e^{-\phi}dV_{\omega}.
\end{align*}

\end{proof}

It follows from Example \ref{exa Dr} that  $c_{\DD_r}(0)=\frac{1}{r}$.  Then Theorem \ref{thm optimal l2 ext} implies:

 \begin{corollary}\label{cor optimal l2 extension}
    Let $D\subset\CC$ be a domain, $E$  a trivial holomorphic vector bundle and $h$ a continuous Hermitian metric on $E$.  If $h$ is $L^2$-optimal, then $h$ satisfies the optimal $L^2$-extension condition.
 \end{corollary}
 
By the above lemmas and Corollary \ref{cor optimal l2 extension}, we thus resolve Conjecture \ref{conj DNW-Ina}:
\begin{theorem}[=Theorem \ref{thm DNW-Ina}]\label{thm DNW-Ina'}
  Every continuous $L^2$-optimal Hermitian metric on a holomorphic vector bundle is Griffiths semi-positive.
\end{theorem}

\end{document}